\documentclass[a4paper,10pt]{article}
\usepackage{mathtext}
\usepackage[T1,T2A]{fontenc}
\usepackage[cp1251]{inputenc}
\usepackage[english]{babel}
\usepackage{amsmath}
\usepackage{amsfonts}
\usepackage{amssymb}
\usepackage{mathrsfs}
\usepackage{amsthm}

\usepackage{titling}

\usepackage{url}

\usepackage{graphicx}
\usepackage{wrapfig}

\usepackage{enumerate}

\usepackage{color}
\usepackage{euscript}

\usepackage{cleveref}

\newcommand{\INT}{\mathcal{I}}
\newcommand{\T}{\mathbb{T}}

\DeclareMathOperator{\interior}{int}

\newcommand{\Class}{\boldsymbol{A}}

\newcommand{\ClassC}{\boldsymbol{A}^{\circ}}

\newcommand{\BMO}{\mathrm{BMO}}

\newcommand{\eps}{\varepsilon}

\DeclareMathOperator{\cl}{cl}

\DeclareMathOperator{\conv}{conv}
\newcommand{\E}[1]{\mathbb{E}_{#1}}
\newcommand{\av}[2]{\langle {#1}\rangle_{{}_{#2}}}

\renewcommand{\leq}{\leqslant}
\renewcommand{\geq}{\geqslant}

\newcommand{\per}{\text{\tiny \textup{per}}}

\newcommand{\BG}{\mathfrak{B}}

\newcommand{\OmNull}{\Omega_{\mathrm{O}}}
\newcommand{\OmOne}{\Omega_{\mathrm{I}}}
\newcommand{\Om}{\Omega}
\newcommand{\tOm}{\tilde{\Om}}

\newcommand{\tOmOne}{\tilde{\Omega}_{\mathrm{I}}}

\DeclareMathOperator{\ch}{ch}

\newcommand{\OmConv}{\Om_{\mathrm{conv}}}

\newcommand{\Leb}{\mathfrak{L}}

\newcommand{\R}{\mathbb{R}}
\newcommand{\Z}{\mathbb{Z}}

\newcommand{\Set}[2]{\Big\{{#1}\;\Big|\,{#2}\Big\}}
\newcommand{\set}[2]{\{{#1}\;|\,{#2}\}}

\newcommand{\eq}[1]{\begin{equation}{#1}\end{equation}}

\newcommand{\alg}[1]{\begin{align}{#1}\end{align}}

\newcommand{\Eeqref}[1]{\stackrel{\scriptscriptstyle{\eqref{#1}}}{=}}

\newcommand{\F}{\mathfrak{S}}

\newtheorem{Le}{Lemma}[section]

\newtheorem{St}[Le]{Proposition}
\newtheorem{Th}[Le]{Theorem}
\newtheorem{Cor}[Le]{Corollary}
\newtheorem{Rem}[Le]{Remark}

\numberwithin{equation}{section}

\begin{document}
\author{Egor Dobronravov\thanks{Supported by the Russian Science Foundation grant 19-71-10023.} \and Dmitriy~Stolyarov\thanksmark{1} 
\and Pavel~Zatitskii}

\title{New Bellman induction and a weak version of $\BMO$}
\maketitle
\begin{abstract}
We enlarge the area of applicability of the Bellman function method to estimates in the spirit of the John--Nirenberg inequality abandoning certain convexity assumptions. As an application, we consider a characteristic of a function that is much smaller than the~$\BMO$ norm, but whose finiteness leads to the exponential integrability of the function.
\end{abstract}

\section{Introduction}\label{S1}

\subsection{Bellman induction on $\BMO$}\label{s11}
Let~$\varphi\colon [0,1]\to \R$ be a square summable function. Denote by~$\av{\varphi}{E}$ its average~$|E|^{-1}\int_E \varphi(x)\,dx$ over a Borel set~$E$ of finite non-zero Lebesgue measure~$|E|$. The quadratic~$\BMO([0,1])$ norm of~$\varphi$ is then defined by the formula
\eq{
\|\varphi\|_{\BMO}^2 = \sup\limits_{\genfrac{}{}{0pt}{-2}{I \text{\tiny\,- subinterval}}{\text{of }[0,1]}} \av{(\varphi - \av{\varphi}{I})^2}{I}.
} 
The famous John--Nirenberg inequality~\cite{JN1961} says that a function with finite~$\BMO$ norm is exponentially integrable. There are several sharp versions of this inequality (see~\cite{Dobronravov2023},~\cite{Korenovski1990},~\cite{Lerner2013},~\cite{SV2011}, \cite{SV2012},~\cite{SV2016}, and~\cite{VV2014}). The ones that include the specific quadratic norm above are usually proved by the so-called Bellman function method. The application of the method to this circle of problems is based on a splitting lemma that appeared for the first time in~\cite{Vasyunin2003} (see Lemma $4$ there) and was used in the search for sharp constants in the reverse H\"older inequalities for Muckenhoupt weights. Let us provide the formulation for the~$\BMO$ case.

Fix some function~$\varphi \in \BMO$. For any subinterval~$J \subset [0,1]$, consider a point~$x_J$ in~$\R^2$,
\eq{
x_J = \Big(\av{\varphi}{J},\av{\varphi^2}{J}\Big).
}
If~$\|\varphi\|_{\BMO} \leq \eps$, then~$x_J \in P_\eps$ for any~$J$, where
\eq{\label{ParabolicStrip}
P_\eps = \set{x\in \R^2}{x_1^2 \leq x_2 \leq x_1^2 + \eps^2}.
}
Note that the latter set, called a parabolic strip, is not convex.
\begin{Le}[Vasyunin's splitting lemma]\label{Lemma5}
For any~$\delta > 0$ and any~$\varphi$ with~$\|\varphi\|_{\BMO}\leq \eps$ there exists a splitting~$[0,1] = I_+\cup I_-$\textup,~$I_\pm$ being subintervals of~$[0,1]$ whose interiors do not intersect such that
\eq{
\big[x_{I_+},x_{I_-}\big]\subset P_{\eps + \delta}
}
and the splitting is regular\textup:~$|I_+|/|I_{-}| \in (\nu,\nu^{-1})$\textup, where~$\nu>0$ is a small number that depends on~$\eps$ and~$\delta$ only. 
\end{Le}
Here and in what follows, by writing~$[P,Q]$ we mean the segment connecting the points~$P$ and~$Q$. By the additivity of the integral,~$x_{[0,1]} \in \big[x_{I_+},x_{I_-}\big]$. Note that the endpoints of the latter segment lie in~$P_\eps$. However, since this domain is not convex, this does  not imply the whole segment lies inside~$P_\eps$. The lemma says one can find a splitting for which it almost does. The John--Nirenberg inequality may be proved by inductive application of the lemma with the help of a certain function on~$P_\eps$. We will sketch this argument in Subsection~\ref{s22}.

The starting point of our research was to find a substitute for Vasyunin's lemma that might work for~$\BMO$ on a square (this might have applications to better bounds for the exponential integrability of~$\BMO$ functions in higher dimension, where the sharp constants are not known, see~\cite{CSS2011} for this). Though our progress in this direction is scant, we managed to find a strengthening of the lemma that is, though still being highly one-dimensional, provides some new information on~$\BMO$-type functions and estimates in the spirit of John and Nirenberg. We will describe this new application in Subsection~\ref{s13} and now turn to the short survey of some results cited above.

\subsection{$\BMO$ and the John--Nirenberg inequality}\label{s12}
The John--Nirenberg inequality in its integral form says there exists~$\eps_0> 0$ such that for any~$\eps \in (0,\eps_0)$ 
\eq{
\av{e^{\varphi-\av{\varphi}{[0,1]}}}{[0,1]}\leq C(\eps),\qquad\text{provided}\quad \|\varphi\|_{\BMO([0,1])} \leq \eps.
}
Here~$C(\eps)$ is a constant that does not depend on the particular choice of~$\varphi$. In~\cite{SV2011}, Slavin and Vasyunin found the sharp values of~$\eps_0$ and~$C(\eps)$.
\begin{Th}[Slavin, Vasyunin, 2011]\label{IntegralJNInterval}
The inequality
\eq{
\av{e^{\varphi-\av{\varphi}{[0,1]}}}{[0,1]} \leq \frac{e^{-\eps}}{1-\eps},\qquad\text{provided}\quad \|\varphi\|_{\BMO([0,1])} \leq \eps,
}
is true\textup, sharp\textup, and attainable for any~$\eps \in [0,1)$.
\end{Th}
See~\cite{Dobronravov2023} for a similar sharp estimate of the quantity~$\av{e^{|\varphi-\av{\varphi}{}|}}{}$ and~\cite{VV2014} for related weak-type estimates. It was shown in~\cite{SZ2021}   that the same estimates are true (this is simple) and sharp (not that simple) for functions on the circle and the line. Let~$\T$ be the circle of radius~$(2\pi)^{-1}$, which we identify with~$\R/\Z$. A function~$\varphi \colon \T \to \R$ may then be naturally identified with its periodic version~$\varphi_\per \colon \R \to \R$ via the formula~$\varphi_\per(t) = \varphi((2\pi)^{-1} e^{2\pi i t})$. Define the~$\BMO$ norm of a function on the circle by the formula
\eq{
\|\varphi\|_{\BMO(\T)} = \|\varphi_\per\|_{\BMO(\R)} =  \sup\limits_{I - \text{\tiny{interval}}} \av{(\varphi_\per - \av{\varphi_\per}{I})^2}{I}.
}
\begin{Th}[Stolyarov, Zatitskii, 2021]
The inequality
\eq{
\av{e^{\varphi-\av{\varphi}{\T}}}{\T} \leq \frac{e^{-\eps}}{1-\eps},\qquad\text{provided}\quad \|\varphi\|_{\BMO(\T)} \leq \eps,
}
is true and sharp for any~$\eps \in [0,1)$.
\end{Th}
There are many versions of~$\BMO$. For example, a dyadic version is quite popular:
\eq{
\|\varphi\|_{\BMO^{\mathrm{dyad}}}^2 = \sup\limits_{I \in \mathbb{D}}  \av{(\varphi - \av{\varphi}{I})^2}{I},
} 
where~$\mathbb{D}$ is the collection of all dyadic subintervals of~$[0,1]$. The functions in this class have slightly worse summability properties than the functions in the classical~$\BMO$ space.
\begin{Th}[Slavin, Vasyunin, 2011]
The inequality
\eq{
\av{e^{\varphi-\av{\varphi}{[0,1]}}}{[0,1]} \leq \frac{e^{-\frac{\eps}{\sqrt 2}}}{2-e^{\frac{\eps}{\sqrt{2}}}},\qquad\text{provided}\quad \|\varphi\|_{\BMO^{\mathrm{dyad}}([0,1])} \leq \eps,
}
is true\textup, sharp\textup, and attainable for any~$\eps \in [0, \sqrt{2}\log 2)$.
\end{Th}
While the condition~$\|\varphi\|_{\BMO^{\mathrm{dyad}}([0,1])} \leq \eps$ is easier to verify than~$\|\varphi\|_{\BMO([0,1])} \leq \eps$, the exponential summability that it  guarantees, is worse:~$\sqrt{2}\log 2 \approx 0.98 < 1$. The forthcoming subsection suggests a version of a weak~$\BMO$ norm, which is easier to compute than the usual~$\BMO$ norm and which leads to almost the same John--Nirenberg inequality.

\subsection{A weak version of $\BMO$}\label{s13}
The idea is to discretize the image of the function, not its domain. Fix a number~$\lambda > 0$ and define the quantity
\eq{\label{StrangeBMO}
[\varphi]_{\BMO([0,1])} = \sup\Set{\av{(\varphi - \av{\varphi}{I})^2}{I}}{I \text{ is a subinterval of } [0,1] \text{ and } \av{\varphi}{I} \in \lambda \Z}.
}
In other words, we restrict the variance of the function only over those intervals whose averages are integer (times~$\lambda$). We are ready to formulate a corollary of our main results.
\begin{Cor}\label{ApplicationInterval}
Fix~$\eps > 0$ and~$\lambda > 0$. Let~$\varphi \colon [0,1]\to \mathbb{R}$ be a square summable function such that~$[\varphi]_{\BMO}< \eps$ and 
\eq{\label{MeanCondition}
\av{|\varphi - \av{\varphi}{[0,1]}|^2}{[0,1]} < \eps^2.
} 
Then~$e^{\mu \varphi}$ is summable\textup, provided
\eq{\label{SummabilityCondition}
\mu < \frac{1}{\lambda}\log\Bigg(\frac{\phantom{-}1+ \sqrt{1 + \frac{4\eps^2}{\lambda^2}}}{-1 + \sqrt{1 + \frac{4\eps^2}{\lambda^2}}}\Bigg).
}
This bound is sharp in the sense that there exists~$\varphi$ satisfying the requirements for which~$e^{(\mu+\delta) \varphi}$ is not summable for the limiting value of~$\mu$ and any~$\delta > 0$.
\end{Cor}
An elementary computation shows that the bounds in the corollary converge to the bounds in Theorem~\ref{IntegralJNInterval} as~$\lambda \to 0$. So, at least in this aspect, the quantity~\eqref{StrangeBMO} is more flexible than the dyadic~$\BMO$ norm. On the other hand, it is worse; for example, it is not a norm. The condition~\eqref{MeanCondition} is important for the proof and we do not know whether it is necessary for the result; though for some similar results it is (see Remark~\ref{FirstTheoremInterval} at the end of Subsection~\ref{s21}). It may be abandoned in the case where the function~$\varphi$ is defined on the circle. We extend the definition of the quantitiy~\eqref{StrangeBMO} to the case of the circle in a natural way:
\eq{\label{StrangeBMOCircle}
[\varphi]_{\BMO(\T)} = \sup\Set{\av{(\varphi_\per - \av{\varphi_\per}{I})^2}{I}}{I \text{ is a subinterval of } \R \text{ and } \av{\varphi_\per}{I} \in \lambda \Z}.
}

\begin{Cor}\label{ApplicationCircle}
Fix~$\eps > 0$ and~$\lambda > 0$. Let~$\varphi \colon \T\to \mathbb{R}$ be a square summable function such that~$[\varphi]_{\BMO}< \eps$. Then~$e^{\mu \varphi}$ is summable\textup, provided~\eqref{SummabilityCondition} holds.
This bound is sharp in the sense that there exists~$\varphi$ such that~$[\varphi]_{\BMO} < \eps$\textup, however\textup,~$e^{(\mu+\delta) \varphi}$ is not summable for the limiting value of~$\mu$ and any~$\delta > 0$.
\end{Cor}

There are several results that describe weak conditions leading to the classical~$\BMO$. Mainly, they describe the functions~$h\colon \R\to (0,\infty)$ that allow the implication
\eq{
\sup\limits_{\genfrac{}{}{0pt}{-2}{I \text{\tiny\,- subinterval}}{\text{of }[0,1]}}\av{h(\varphi - \av{\varphi}{I})}{I} < \infty \qquad \Longrightarrow\qquad \varphi \in \BMO([0,1]).
}
We refer to the early paper~\cite{John1965} and also~\cite{CPR2022},~\cite{LSSVZ2015},~\cite{LY1984},~\cite{LT1987}, and~\cite{Stromberg1979}. We have not found results close to the corollaries above in the literature.

In the following section, we will introduce more general classes of functions and formulate the results in the spirit of Corollaries~\ref{ApplicationInterval} and~\ref{ApplicationCircle} in their generality. Those are the main results of our paper (Theorems~\ref{Th1} and~\ref{Th2}). We will also derive the corollaries from Theorems~\ref{Th1} and~\ref{Th2} in Subsection~\ref{s24}. Sections~\ref{S3} and~\ref{S4} contain the proofs of the main results. The appendix contains supplementary material: a short discussion of the Lebesgue differentiation theorem in Section~\ref{AS1} and the proofs of auxiliary geometric statements in Section~\ref{AS2}.

\section{Statement of results}\label{S2}
\subsection{Function classes}\label{s21}
The~$\BMO$ functions and functions with finite quantity~\eqref{StrangeBMO} are related to certain vectorial functions and martingales. We will follow the notation of~\cite{SZ2022}, which generalizes~\cite{IOSVZ2015_bis},~\cite{SZ2016}, and~\cite{SVZ2015}. Let~$\OmNull$ be an open strictly convex non-empty subset of~$\R^d$. This is the set where the averages of our vectorial functions martingales are allowed to attain values. The closed set~$\OmOne\subset \OmNull$ is the forbidden set for the averages. The introduction of such a set imposes restrictions on mean oscillations. For example, the parabolic strip~$P_\eps$ given in~\eqref{ParabolicStrip} is the set theoretic difference of~$\cl\OmNull$ and~$\interior\OmOne$ given by
\eq{\label{ParabolicStrip}
\OmNull = \set{x\in \R^2}{x_1^2 < x_2};\qquad \OmOne = \set{x\in \R^2}{x_1^2 + \eps^2 \leq x_2}.
}
The quantity~\eqref{StrangeBMO} is generated by the domain that is the difference of~$\cl\OmNull$ and~$\OmOne$, where
\eq{\label{DomainForStrangeBMO}
\OmNull = \set{x\in \R^2}{x_1^2 < x_2};\quad\OmOne =  \bigcup\limits_{n\in\mathbb{Z}}\Set{x\in \R^2}{x_1 = \lambda n, x_2 \geq \lambda^2n^2 + \eps^2}.
}
In other words, if~$[\varphi]_{\BMO}< \eps$, then~$\av{\psi}{J}\notin \OmOne$ for any interval~$J$, where~$\psi = (\varphi,\varphi^2)$, and, vice versa, if the averages of~$\psi$ avoid the set~$\OmOne$, then~$[\varphi]_{\BMO} \leq \eps$. See Fig.~\ref{Fig:StrangeBMO} for visualization. 

\begin{figure}[h!]
\centerline{\includegraphics[height=8cm]{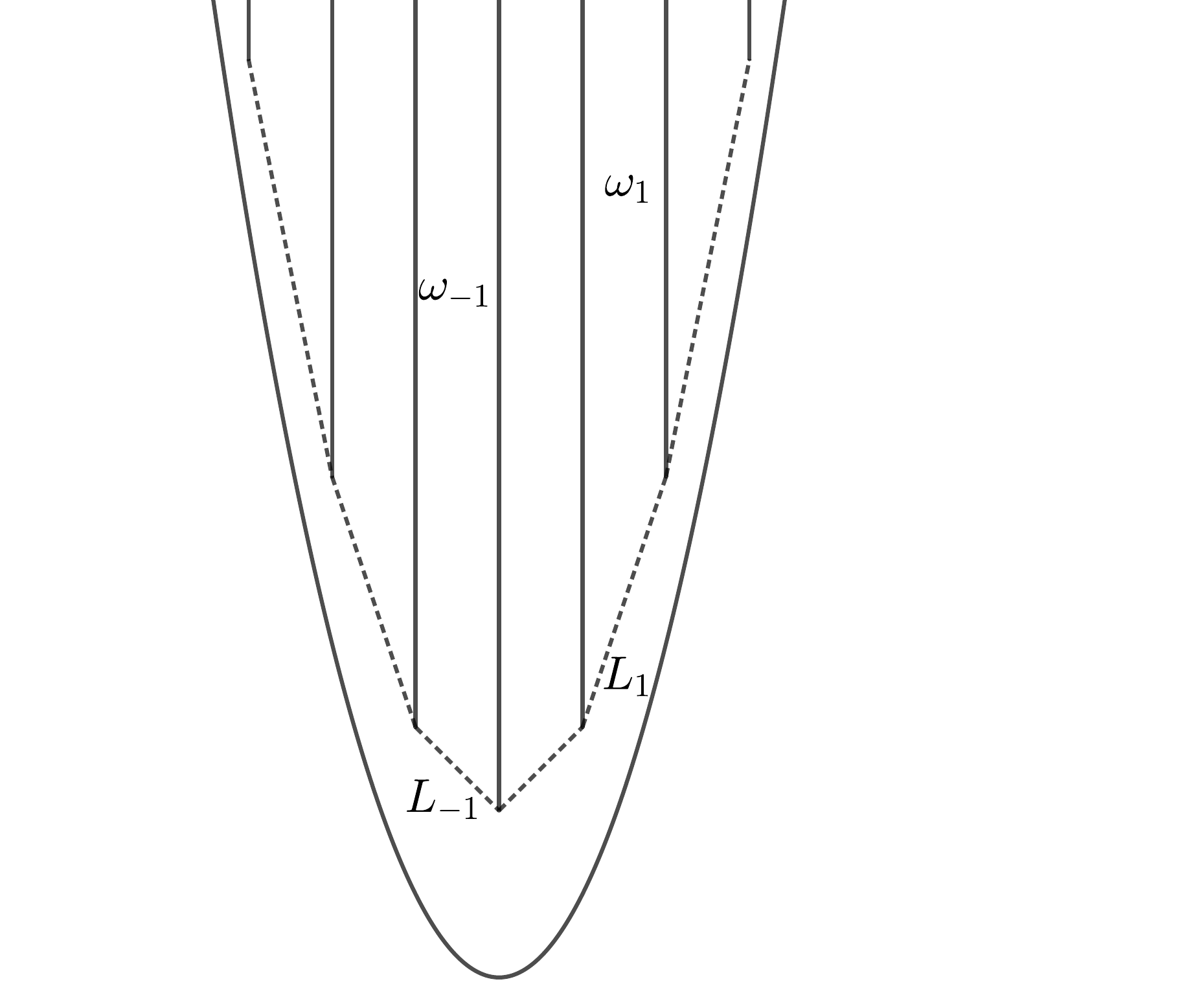}}
\caption{The domain related to the `norm'~\eqref{StrangeBMO}.}
\label{Fig:StrangeBMO}
\end{figure}

In the previous papers, the forbidden set~$\OmOne$ was strictly convex. Now we wish to get rid of this assumption. We postulate the following `axioms': 
\alg{
\label{First}&\text{1) the boundary of the set } \conv \OmOne \text{ does not contain rays;}\\
\label{Second}&\text{2) the closure of } \conv\OmOne \text{ lies inside } \OmNull;\\
\label{Third}&\text{3) the sets } \OmNull \text{ and } \conv\OmOne \text{ have congruent maximal inscribed cones, $\interior \conv \OmOne \ne \varnothing$;}\\
\label{Fourth}&\text{4) the set } (\interior \conv\OmOne)\setminus \OmOne~ \text{ is a locally finite union of its connectivity components } \omega_j;\\
\label{Fifth}&\text{5) for any } j \text{ there exists a supporting plane } L_j  \text{ to } \conv \OmOne \text{ that contains } E_j = \partial \omega_j \setminus \OmOne.
}
The fourth requirement means that for any bounded domain in~$\R^d$ there is only a finite number of~$\omega_j$ that intersect this domain. 
The reader may verify that the domain~\eqref{DomainForStrangeBMO} satisfies all five requirements. We provide a brief comment on them. The first requirement~\eqref{First}, in particular, implies~$\conv \OmOne$ is closed,  by the following lemma.
\begin{Le}\label{ClosureConvexHull}
Let~$F \subset \R^d$ be a closed set such that~$\partial \conv F$ does not contain rays. Then\textup,~$\cl \conv F = \conv F$.
\end{Le}
It is interesting to find an infinite-dimensional version of the previous lemma in connection with the Krein--Milman theorem. We postpone the proof of Lemma~\ref{ClosureConvexHull} till Subsection~\ref{AS2}. 

The second requirement~\eqref{Second} is quite natural when~$d=2$, where most interesting applications occur. For larger~$d$ its absence may lead to unexpected interesting effects, of which we have very poor understanding. The second assumption is violated for three-dimensional domains that appear in~\cite{SVZ2020} and~\cite{VZZ2022}.

The third assumption~\eqref{Third} is also necessary for the theory when~$d=2$. In the papers~\cite{SZ2016} and~\cite{SZ2022} it was called the cone condition. If it is violated when~$d=2$ many natural Bellman functions are equal to~$-\infty$ on parts of the domain (see Lemmas~$2.4$ and~$4.8$ in \cite{SZ2022}). The assumptions~\eqref{First},~\eqref{Second}, and~\eqref{Third} work with the convex hull~$\conv \OmOne$ rather than with~$\OmOne$ itself. The fourth assumption~\eqref{Fourth} simplifies the structure of the forbidden set and~\eqref{Fifth} is extremely important for the proof. Without this assumption the results are not valid in any sense (see Remark~\ref{RemarkConnectivity} below).
\begin{Rem}\label{StructuralRemark}
Assume the conditions~\eqref{First}\textup,~\eqref{Second}\textup,~\eqref{Third}\textup,~\eqref{Fourth}\textup, and~\eqref{Fifth} are satisfied. Then\textup,
\eq{\label{BoundaryStructure}
\partial \conv\OmOne \setminus \OmOne = \bigcup_j\big(\partial \omega_j \setminus\OmOne\big)  \subset \bigcup_j L_j.
}
The equality above might need justification. 

The inclusion~$\partial \conv\OmOne \setminus \OmOne \subset \bigcup_j\big(\partial \omega_j \setminus\OmOne\big)$ may be justified as follows. If~$x$ belongs to~$\partial \conv\OmOne$\textup, but does not belong to~$\OmOne$\textup, there exists a sequence~$\{x_n\}_n$ lying in~$\interior\conv\OmOne\setminus\OmOne$ that tends to~$x$ \textup(here we use that~$\interior\conv\OmOne$ is non-empty by~\eqref{Third}\textup). By~\eqref{Fourth}\textup, there exists~$j$ and a subsequence~$\{x_{n_k}\}_k$ such that~$x_{n_k}\in \omega_j$. This proves the desired inclusion. The reverse inclusion follows from definitions.

Let~$L_j^+$ be the open half-space generated by~$L_j$ that contains~$\interior \conv\OmOne$. Then\textup, for any~$x\in E_j$ there exists a small radius~$r$ such that~$B_r(x)\cap L_j^+\subset \omega_j$. 

To prove this\textup, let~$r$ be so small that~$B_r(x) \cap \OmOne = \varnothing$. Assume there exists~$y\notin \omega_j$ that belongs to~$B_r(x)\cap L_j^+$. Since~$x\in \partial \omega_j$\textup, there also exists a point~$z\in \omega_j$ inside~$B_r(x)\cap L_j^+$. Consequently\textup, there is a point of~$\partial \omega_j$ on~$[y,z]$\textup, which contradicts~\eqref{Fifth}. Thus\textup, any~$y\in B_r(x) \cap L_j^+$ belongs to~$\omega_j$.
\end{Rem}

It will be convenient to use the notation
\eq{
\Om = \cl\OmNull\setminus \OmOne;\quad \OmConv = \cl\OmNull\setminus \conv\OmOne.
}
Note that unlike papers~\cite{SZ2016} and~\cite{SZ2022}, those sets are not closed.
We finish the discussion of the conditions imposed on~$\Om$ with a simple sufficient condition.
\begin{St}
Let~$d=2$ and let~$\Om$ be generated by~$\OmOne$ and~$\OmNull$ that satisfy requirements~\eqref{First}\textup,~\eqref{Second}\textup,~\eqref{Third}\textup, and~\eqref{Fourth}. If~$\OmOne$ is linearly connected\textup, then~\eqref{Fifth} also holds true.
\end{St}
\begin{proof}
Consider~$\omega_j$ for some~$j$. Let~$x\in \partial \omega_j\setminus \OmOne$. Then,~$x\in \partial \conv\OmOne$. Let~$\ell$ be a supporting line to~$\conv\OmOne$ at the point~$x$. It suffices to prove that it is unique and does not depend on the particular choice of~$x\in \partial \omega_j\setminus \OmOne$.

Since the boundary of~$\conv\OmOne$ does not contain rays,~$\ell\cap \conv\OmOne$ is a bounded closed set. Its leftmost and rightmost (inside the line~$\ell$) points~$P$ and~$Q$ belong to~$\OmOne$. By the assumption that~$\OmOne$ is linearly connected, there exists a path~$\gamma$ connecting~$P$ to~$Q$ inside~$\OmOne$. Therefore,~$\omega_j$ lies inside the contour bounded by~$[P,Q]$ and~$\gamma$. Thus, any point in~$\partial \omega_j\setminus \OmOne$ lies in~$\ell$ by~\eqref{BoundaryStructure}. 
\end{proof}

Now we are ready to define the function classes. We will be considering summable functions that map~$[0,1]$ or~$\T$ into~$\partial \OmNull$. The averages of such functions are defined coordinate-wise. Define the function classes
\begin{align}
\label{Class}\Class_\Omega = \Set{\psi\in L_1([0,1],\R^d)}{\forall x\in [0,1]\ \psi(x)\in \partial \OmNull, \forall J\ \text{subinterval of}\ [0,1]\quad \av{\psi}{J}\notin \OmOne};\\
\ClassC_\Omega = \Set{\psi\in L_1(\T,\R^d)}{\forall x\in\T\ \psi(x)\in \partial \OmNull, \forall J\ \text{subinterval of}\ \R\quad \av{\psi_\per}{J}\notin \OmOne}.
\end{align}
We are ready to formulate our first main result.
\begin{Th}\label{Th1}
Assume the domains satisfy~\eqref{First}\textup,~\eqref{Second}\textup,~\eqref{Third}\textup,~\eqref{Fourth}\textup, and~\eqref{Fifth}. Then\textup, for any~$\psi \in \ClassC_\Omega$\textup,~$\av{\psi}{\T}\notin \interior\conv\OmOne$.
\end{Th}
\begin{Rem}\label{FirstTheoremInterval}
A similar assertion for functions on the interval is invalid. Namely\textup, one may consider the domain~\eqref{DomainForStrangeBMO} and a function~$\psi\colon [0,1]\to \set{x\in \R^2}{x_1^2 = x_2}$ such that~$\psi \in \Class_\Omega$ and~$\av{\psi}{[0,1]} \in \interior\conv\OmOne$. See Fact~$9.2$ in~\textup{\cite{SZ2022}} for a similar function attaining three values only.
\end{Rem}
\begin{Rem}\label{InvarianceRemark}
We may replace the interval~$[0,1]$ with an arbitrary interval in our formulations without any harm. What is more\textup, this independence of the interval \textup(the homogeneity of the problem\textup) is crucially important.
 \end{Rem}

\subsection{Locally concave functions}\label{s22}
Let~$\omega \subset \R^d$ be a set. A function~$G\colon \omega \to \R\cup \{-\infty,+\infty\}$ is called locally concave provided its restriction to any convex subset of~$\omega$ is concave. We adopt standard convention about infinite values of convex and concave functions (see Section~$4$ in~\cite{Rockafellar1970}). Locally concave functions play an important role in the subject. We will show this on the standard~$\BMO$ example. Assume~$\Omega$ is the parabolic strip~$P_\eps$ defined in~\eqref{ParabolicStrip} and~$G\colon P_\eps \to \R$ is a finite locally concave function. Let~$\|\varphi\|_{\BMO([0,1])}\leq \eps$ and let~$\psi = (\varphi,\varphi^2)$ be the corresponding vectorial function. We apply Lemma~\ref{Lemma5} and obtain the splitting~$[0,1] = I_+\cup I_-$. If we disregard the number~$\delta$ or assume~$G$ is locally concave on a slightly wider parabolic strip, then local concavity implies
\eq{
G(x) \geq \frac{|I_+|}{|I|}G(x_+) + \frac{|I_-|}{|I|}G(x_-),\quad x_{\pm} = \av{\psi}{I_{\pm}},
} 
since~$G|_{[x_+,x_-]}$ is concave. We apply Lemma~\ref{Lemma5} to each of the intervals~$I_\pm$ in the role of~$[0,1]$ and obtain a good splitting of each of them. We thus get a splitting of~$[0,1]$ into four intervals~$J$ such that
\eq{\label{InductionResult}
G(x)\geq \sum_{J} \frac{|J|}{|I|}G(x_J),\qquad x_J = \av{\psi}{J}.
}
We continue this process and, at step~$n$, obtain a splitting of~$[0,1]$ into~$2^n$ intervals~$J$ with the same inequality. Note that the length of each interval~$J$ is at most~$(1-\nu)^{n}$, where~$\nu$ is the splitting parameter from Lemma~\ref{Lemma5}. The sum on the right of~\eqref{InductionResult} converges to~$\int_0^1G(\psi(t))\,dt$ (here we skip the limiting argument based on the Lebesgue differentiation theorem). If we assume that~$G(t,t^2) = e^{t}$, which might be thought of as the boundary condition for a locally concave function~$G$, then we obtain the inequality
\eq{
G(x) \geq \int\limits_0^1e^{\varphi(t)}\,dt.
}
This estimate might be thought of as a version of Jensen's inequality on non-convex domain, if we rewrite it in the form~$G(\av{\psi}{[0,1]})\geq \av{G(\psi)}{[0,1]}$.
The procedure of splitting an interval into subintervals and keeping track of the quantities~\eqref{InductionResult} is informally called the \emph{Bellman induction}.

Thus, a locally concave function on a parabolic strip with the boundary conditions~$G(t,t^2) = e^t$ leads to an estimate in the spirit of John and Nirenberg. Surprisingly, the minimal possible function~$G$ leads to sharp estimates, in particular, to Theorem~\ref{IntegralJNInterval}. Let us introduce this object and for this return to the generality of the sets~$\OmOne$ and~$\OmNull$. Let~$f\colon \partial \OmNull\to \R$ be a locally bounded function. Consider the set
\eq{
\Lambda_{\Omega,f} = \Set{G\colon \Omega \to \R\cup \{-\infty,+\infty\}}{G\ \text{is locally concave on~$\Om$ and}\ \forall x\in \partial \OmNull\quad G(x) = f(x)}
}
and the pointwise minimal function
\eq{
\BG_{\Omega,f}(x) = \inf\Set{G(x)}{G\in \Lambda_{\Omega,f}},\qquad x\in \Omega.
}
Note that~$\BG_{\Omega,f} \in \Lambda_{\Omega,f}$. This function will play an important role in the subject.

The computation of~$\BG_{\Omega,f}$ is a separate problem on the intersection of differential geometry and PDE. In the case where~$d=2$ and~$\OmOne$ is strictly convex, it was solved in~\cite{ISVZ2023} (see also earlier papers~\cite{IOSVZ2015} and~\cite{ISVZ2018} for particular cases).

\begin{Rem}\label{DilationInvariance}
The Bellman induction is based on the dilation invariance of the problem in question. One may define the~$\BMO$ space on an arbitrary interval in a similar manner and study the John--Nirenberg inequality. The sharp constants will be the same since the problem is dilation invariant by Remark~\textup{\ref{InvarianceRemark}}.
\end{Rem}

\subsection{New splitting lemma}\label{s23}
\begin{Le}\label{EvenNewer5}
Assume the domains~$\OmNull$ and~$\OmOne$ satisfy requirements~\eqref{First}\textup,~\eqref{Second}\textup,~\eqref{Third}\textup,~\eqref{Fourth}\textup, and~\eqref{Fifth}. For every function~$\psi \in \Class_\Omega$ with~$x=\av{\psi}{[0,1]}\notin \interior\conv\OmOne$ there exists a non-trivial splitting~$[0,1] = I_+\cup I_-$ into two intervals such that
\eq{
\Big[\av{\psi}{I_+},\av{\psi}{I_-}\Big]\subset \cl\OmNull\setminus \interior \conv \OmOne.
}
\end{Le}
By a non-trivial splitting we mean that the intervals~$I_\pm$ have non-empty interior, intersect by endpoints, and cover~$[0,1]$. Though this lemma is a generalization of Lemma~\ref{Lemma5}, there are some differences. We do not have any widening of the domain (in Lemma~\ref{Lemma5} there is one), and also no uniform control of the smallness of~$I_+$ and~$I_-$. This comes from the fact that the domain~$\Omega$ is no longer closed. In a neighborhood of~$\OmOne$, it is open. As a result, there will be some difficulties with applications of Lemma~\ref{EvenNewer5}. We will need to modify the Bellman induction to a reasoning we call \emph{transfinite Bellman induction}. See Subsection~\ref{s42} for details.

We are ready to formulate our second main result that implies Corollary~\ref{ApplicationInterval}.
\begin{Th}\label{Th2}
Assume~$\OmNull$ and~$\OmOne$ satisfy the requirements~\eqref{First}\textup,~\eqref{Second}\textup,~\eqref{Third}\textup,~\eqref{Fourth}\textup, and~\eqref{Fifth} and the function~$f\colon \partial \OmNull\to \R$ is continuous and bounded from below. The inequality 
\eq{\label{MainIneq}
\av{f(\psi)}{[0,1]}\leq \BG_{\OmConv,f}(\av{\psi}{[0,1]})
}
holds true for any~$\psi \in \Class_{\Om}$ such that~$\av{\psi}{[0,1]}\notin \interior\conv\OmOne$.
\end{Th}
\begin{Rem}\label{RemarkConnectivity}
The assertion of the theorem is invalid without the requirement~\eqref{Fifth}. The domain given in Section~$9$ of~\textup{\cite{SZ2022}} provides an example. More specifically\textup, let~$\OmNull = \set{x\in \R^2}{x_1^2 + x_2^2 \leq 1}$ and let~$\OmOne$ be the union of two disks with centers~$(1/2,0)$ and~$(-1/2,0)$ and radii~$0.4$. Let~$f$ be given by the formula~$f(x) = -|x_1|$ and let~$\psi\colon [0,1]\to \R$ attain two values:
\eq{
\psi(t)=
\begin{cases}
(0,-1),\qquad &t\in [0,0.9);\\
(0,1),\qquad &t\in [0.9,1].
\end{cases}
}
Clearly\textup,~$\psi \in \Class_\Om$\textup,~$x=\av{\psi}{[0,1]}\notin \interior\conv\OmOne$\textup, and~$\av{f(\psi)}{[0,1]} = 0$. However, one may see~$\BG_{\OmConv,f}(x)<0$ in this case.
\end{Rem}
\begin{Rem}
The condition that~$f$ is continuous in Theorem~\textup{\ref{Th2}} seems superfluous. At least\textup, it may be replaced with the assumption that~$f$ is lower semicontinuous without changing the proofs much.
\end{Rem}
\begin{Rem}
The domain~$\conv\OmOne$ is not strictly convex. Thus\textup, it does not fall under the scope of the theory of~\textup{\cite{ISVZ2023}} even in the case~$d=2$. However\textup, the heuristics provided by that paper makes the computation of~$\BG_{\OmConv,f}$ available as we will show in the following subsection. 
\end{Rem}

\subsection{Proofs of the corollaries}\label{s24}
\begin{proof}[Proof of Corollary~\textup{\ref{ApplicationInterval}}]
By Theorem~\ref{Th2}, to prove the estimate, it suffices to show the finiteness of the minimal locally concave function~$\BG_{\OmConv,f}$ for~$\Om$ given in~\eqref{DomainForStrangeBMO}
and the boundary values~$f(x) = e^{\mu x_1}$. Call this function simply~$\BG$. We will suggest a formula for~$\BG$. Before that we note the natural homogeneity identity
\begin{figure}[h!]
\centerline{\includegraphics[height=9cm]{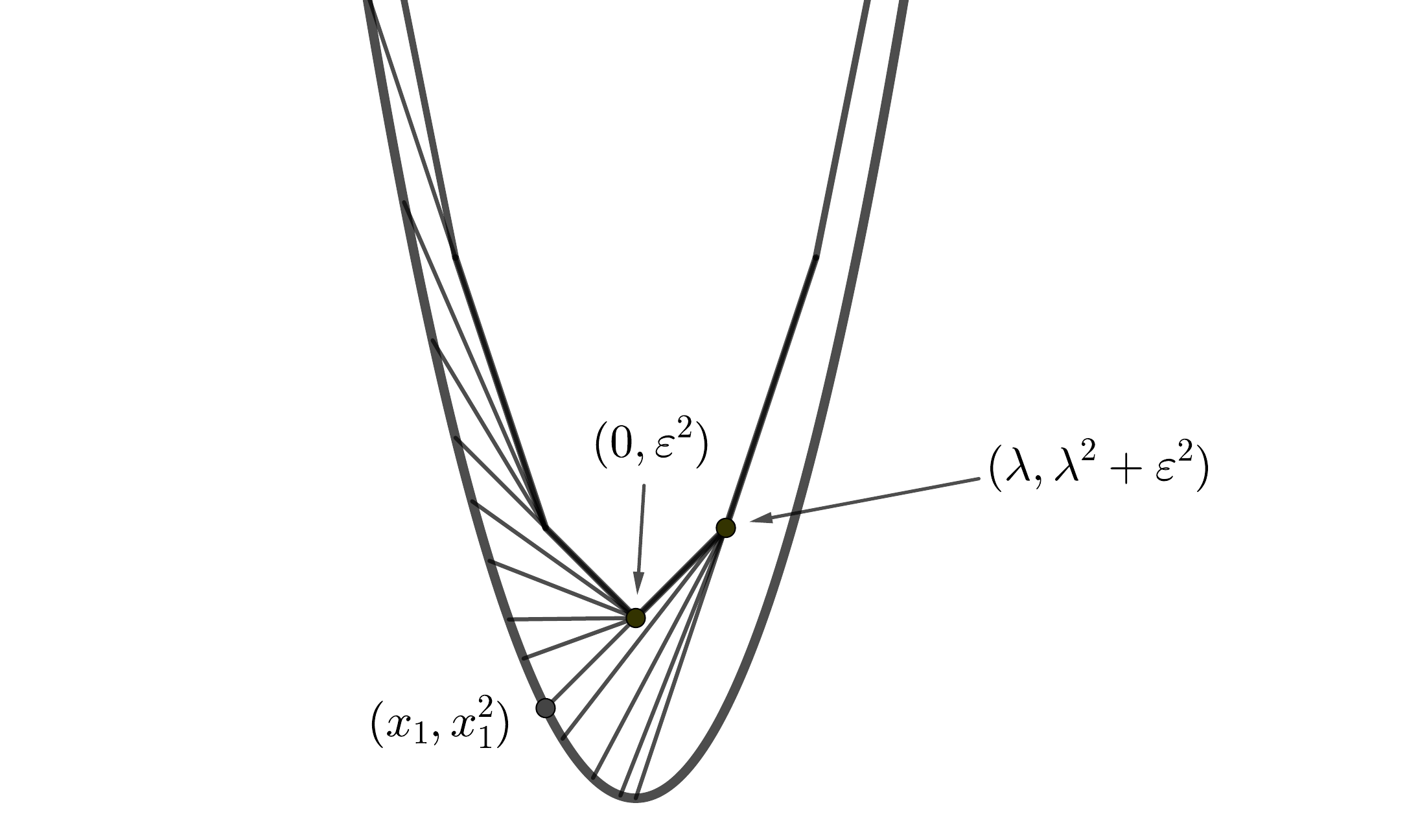}}
\caption{Foliation of the Bellman function}
\label{Fig_Fol}
\end{figure}
\eq{\label{Homogeneity}
\BG(x_1+\lambda,x_2 + 2x_1\lambda + \lambda^2) = e^{\lambda\mu}\BG(x_1,x_2),\qquad x\in \Omega.
}
Observe that the curve~$t\mapsto (t,t^2,e^{\mu t})$ has positive torsion. The theory from~\cite{ISVZ2023} (Section~$3.3$ of that paper) suggests~$\BG$ is affine along the segments in~$\OmConv$  drawn on Fig.~\ref{Fig_Fol}. The splitting of~$\OmConv$ is called a \emph{foliation}. Here it consists of segments connecting points on the boundary~$\partial\OmNull$ with the points~$(\lambda n, \lambda^2 n^2 + \eps^2)$. If the boundary~$\partial \conv\OmOne$ was smooth, the foliating segments would be tangent to it. Though the theory from~\cite{ISVZ2023} applies to domains with sufficiently smooth boundary, in the case where the boundary curve has positive torsion it may be also applied when the upper boundary is non-smooth. The reason is that the torsion condition is imposed on the lower boundary. Let~$(x_1,x_1^2)$ be the point of intersection of the line passing through~$(0,\eps^2)$ and~$(\lambda,\lambda^2 + \eps^2)$ with the fixed boundary. Then,
\eq{
x_1=\frac{\lambda}{2} - \sqrt{\frac{\lambda^2}{4} + \eps^2}.
}
Using that~$\BG$ is affine on the segment~$[(x_1,x_1^2), (\lambda,\lambda^2 + \eps^2)]$ together with~\eqref{Homogeneity}, we see that
\eq{
\BG(\lambda n, \lambda^2 n^2 + \eps^2) = \frac{\lambda}{\lambda/2 + \sqrt{\lambda^2/4 + \eps^2} + (\lambda/2 - \sqrt{\lambda^2/4 + \eps^2}) e^{\lambda \mu}} e^{\mu(\lambda n + \lambda/2 - \sqrt{\lambda^2/4 + \eps^2})},\ n \in\mathbb{Z}.
}
The condition~\eqref{SummabilityCondition} is needed for this formula to provide a positive value. The values at other points in~$\Omega$ are restored by linearity from the formula above and the foliation. We leave the verification of the local concavity of~$\BG$ to the reader. The summability of~$e^{\mu \varphi}$ is proved.

To prove the sharpness, we consider a specific function~$\varphi$; the construction is again suggested by the general theory in Chapter~$5$ of~\cite{ISVZ2023}. To this end, we set
\eq{
a= 1 - \frac{\lambda}{\lambda/2 + \sqrt{\lambda^2/4 + \eps^2}}
}
and define the function~$\varphi\colon [0,1]\to\R$ as
\eq{
\varphi(t) = \Big(n+\frac12\Big)\lambda - \sqrt{\frac{\lambda^2}{4} + \eps^2},\qquad t\in [a^{n+1},a^n],\quad n\geq 0.
}
One may see that~$\av{e^{\mu \varphi}}{[0,1]} = \BG(0,\eps^2)$ and that the curve
\eq{
\gamma(t) = \big(\av{\varphi}{[0,t]}, \av{\varphi^2}{[0,t]}\big)
}
goes along the upper boundary of~$\OmConv$ from~$+\infty$ downto~$(0,\eps^2)$. Thus, by the theory in Chapter~$5$ of~\cite{ISVZ2023} (namely, by Corollary~$5.1.4$), we have~$\varphi \in \Class_{\OmNull \setminus \interior \conv\OmOne}$. This proves the sharpness of our bounds.
\end{proof}
\begin{proof}[Proof of Corollary~\textup{\ref{ApplicationCircle}}]
To prove the estimate, consider the corresponding vectorial function~$\psi\colon \T\to \R^2$ given by the rule~$\psi(t) = (\varphi(t),\varphi^2(t))$. Then,~$\psi \in \ClassC_{\Om}$, where~$\Om$ is given by~\eqref{DomainForStrangeBMO}. By Theorem~\ref{Th1},~$\av{\psi}{\T}\notin \interior\conv\OmOne$ and thus, we are in position to apply Theorem~\ref{Th2}. This theorem shows~$e^{\mu \varphi}$ is summable, provided~$\mu$ satisfies~\eqref{SummabilityCondition}. 

To prove the sharpness, we simply invoke Lemma~$3.3$ in~\cite{SZ2021} that constructs a function~$\tilde{\psi}\in \ClassC_{\tOm}$ such that~$\av{e^{\mu\tilde{\psi}_1}}{\T} \geq \BG_{\tOm,e^{\mu\cdot}}(\av{\tilde{\psi}}{\T})$. Here~$\tOm$ is a domain formed by the same~$\OmNull$ and a convex set~$\tOmOne$ that is slightly smaller than~$\conv\OmOne$.
\end{proof}

\section{Proof of Theorem~\ref{Th1}}\label{S3}

\subsection{Proof}\label{S31}
Assume the contrary, let~$x = \av{\psi}{[0,1]} \in \interior (\conv \OmOne)\setminus \OmOne$ for some~$\psi \in \ClassC_{\Om}$. 
We will implicitly identify the function~$\psi\colon \T\to \R^d$ with its periodization~$\psi_\per\colon \R\to \R^d$ and hope this will not lead to ambiguity. 

By definition,~$x\in \omega_j$ for some~$j$. Recall the sets~$E_j$ in~\eqref{Fifth} and define a stopping time~$\tau(t)$:
\eq{
\tau(t) = \inf \Set{s\in \R}{s > 0, \quad \av{\psi}{[t,t+s]} \in E_j},\qquad t\in \R.
} 
From now till the end of the proof we omit the index~$j$ in our notation. We will be using a special notion of a Lebesgue point, see Subsection~\ref{AS1} for the definition. First,~$\tau(t) > 0$ for any Lebesgue point~$t$ since the curve~$s\mapsto  \av{\psi}{[t,t+s]}$,~$s\in (0,1]$, tends to~$\psi(t) \in \partial \OmNull$ as~$s\to 0+$ and the set~$E$ is separated from~$\partial \OmNull$ by \eqref{Second}. Second,~$\tau(t) \leq 1$ for a Lebesgue point~$t$, since the said curve connects~$x$ to~$\partial\OmNull$ and, thus, intersects~$E$ (it intersects~$\partial \omega$ by continuity and does not intersect~$\OmOne$ by the requirement~$\psi \in \ClassC_{\Om}$). 

Pick some~$\delta > 0$ and consider the set
\eq{
\tilde{U}_\delta = \set{t\in\R}{t \ \text{is a Lebesgue point of}\ \psi \ \text{ and}\ \tau(t) > \delta}.
}
By the continuity of measure,~$|[0,1]\setminus \tilde{U}_\delta|\to 0$ as~$\delta \to 0$.  
Let~$U_\delta \subset \tilde{U}_\delta$ be a closed~$1$-periodic set such that
\eq{
|[0,1] \setminus U_\delta| \leq 2|[0,1]\setminus \tilde{U}_\delta| + \delta.
}
In particular,~$|[0,1] \setminus U_\delta| \to 0$ as~$\delta \to 0$.

Construct the sequence~$\{t_n\}_{n\in \mathbb{N}\cup \{0\}}$ by the following rule:
\eq{
t_{n+1} = \inf\Set{u\in \R}{u \geq t_n + \tau(t_n)\ \text{and} \ u\in U_\delta};
}
let also~$t_0 = 0$ (without loss of generality, we assume~$0\in U_\delta$). This sequence has three simple properties:
\begin{enumerate}[1)]
\item $\av{\varphi}{[t_n,t_n + \tau(t_n)]}\in E$;
\item $t_{n+1} \geq t_n + \delta$;
\item $[t_n + \tau(t_n), t_{n+1}]\subset \R\setminus U_\delta$.
\end{enumerate}
The second property, in particular, implies that~$t_n\to \infty$. Let~$N$ be a large number. Then,
\eq{\label{Decomposition}
\av{\psi}{[0,t_N]} = \frac{\sum_0^{N-1}\tau(t_n)}{t_N}\sum\limits_{n=0}^{N-1}\frac{\tau(t_n)}{\sum_{k=0}^{N-1}\tau(t_k)}\av{\psi}{[t_n,t_n + \tau(t_n)]} + \frac{1}{t_N}\sum\limits_{n=0}^{N-1}\int\limits_{t_n + \tau(t_n)}^{t_{n+1}}\psi(s)\,ds.
} 
The second summand may be estimated using the third property of the sequence~$\{t_n\}$:
\eq{
\Big|\frac{1}{t_N}\sum\limits_{n=0}^{N-1}\int\limits_{t_n + \tau(t_n)}^{t_{n+1}}\psi(s)\,ds\Big| \leq \frac{1}{t_N}\int\limits_{[0,t_N]\setminus U_\delta}|\psi(s)|\,ds \leq \frac{t_N + 1}{t_N}\mathrm{Er},
}
where~$\mathrm{Er} = \int_{[0,1] \setminus U_\delta}|\psi|$. By the absolute continuity of the integral,~$\mathrm{Er}\to 0$ as~$\delta \to 0$.

Note that the first sum in~\eqref{Decomposition} (disregard the coefficient~$t_N^{-1}\sum_0^{N-1}\tau(t_n)$) belongs to~$\conv E$, which lies in~$L$ by~\eqref{Fifth}. The coefficient, in its turn, can be estimated as
\eq{
\frac{\sum_0^{N-1}\tau(t_n)}{t_N} \geq 1 - \frac{|[0,t_N] \setminus U_\delta|}{t_N} \geq 1-\frac{t_N+1}{t_N}|[0,1]\setminus U_\delta|.
} 
In particular, it tends to one as~$\delta \to 0$. To come to a contradiction, it remains to make the right choice of the parameters. We choose~$\delta$ so small that the distance from~$\lambda x$ to~$L$ is larger than~$2\mathrm{Er}$ for any~$\lambda \in [1- 4\delta, (1-4\delta)^{-1}]$ and then pick sufficiently large~$N$. Then, the inequalities above lead us to contradiction with the fact that the first sum in~\eqref{Decomposition} lies in~$L$ and the whole expression converges to~$x\in \interior \conv\OmOne$. \qed

\subsection{A useful lemma}\label{s32}
The following lemma is needed for the proof of Theorem~\ref{Th2}. The proof of the lemma is based on the same circle of ideas as that of Theorem~\ref{Th1}.
\begin{Le}\label{UsefulLemma}
Let~$\psi \in \Class_{\Om}$ be defined on~$[0,1]$. Let~$\av{\psi}{[0,1]} = x$ belong to~$\omega_j$ for some~$j$. For any~$\eps > 0$ there exists~$t\in (0,\eps)$ such that~$\av{\psi}{[0,t]}\notin \omega_j$.
\end{Le}
\begin{Rem}
The assumption~$x\in\omega_j$ is formally unnecessary. However\textup, it is convenient for applications.
\end{Rem}
The proof is based upon the following lemma (which is a straightforward consequence of Zorn's lemma).
\begin{Le}\label{BabyBesicovitch}
For any family~$\mathcal{I} = \set{I_x}{x\in (0,1]}$ of nonempty semi-intervals~$I_x = (y,x]\subset (0,1]$\textup, there exists a subset~$C\subset (0,1]$ such that~$\cup_{x\in C} I_x = (0,1]$ and whenever~$y,z\in C$ and~$y\ne z$\textup, the intervals~$I_y$ and~$I_z$ are disjoint. 
\end{Le}
In other words, the lemma provides us with disjoint  intervals whose union is~$(0,1]$; this may be thought of as a one-dimensional version of the Besicovitch covering theorem.
\begin{proof}[Proof of Lemma~\textup{\ref{UsefulLemma}}.]
Assume the contrary. Without loss of generality, assume~$\eps = 1$ and~$\av{\psi}{[0,t]} \in\omega$ for any~$t\in (0,1]$. Let~$\mathfrak{L}$ be the set of Lebesgue points of~$\psi$ in~$(0,1]$. Similar to the proof in the previous subsection, we disregard the index~$j$ and write~$\omega$,~$E$, and~$L$ for~$\omega_j$,~$E_j$, and~$L_j$, respectively. For any~$t\in \mathfrak{L}$ there exists a non-empty interval~$I_t = (s,t]$ such that~$\av{\psi}{I_t}\in E$ (since~$\av{\psi}{(0,t]}\in \omega$ and~$\av{\psi}{[s,t]}\to \partial \OmNull$ as~$s\to t-$). Consider a disjoint family~$\{J_i\}$ of open subintervals of~$(0,1]$ that cover~$(0,1]\setminus \mathfrak{L}$ and whose total length does not exceed~$\delta$. Here~$\delta$ is a small number to be chosen later. Any point~$t \notin \mathfrak{L}$ belongs to some~$J_i = (a_i,b_i)$. Set~$I_t = (a_i,t]$.

We apply Lemma~\ref{BabyBesicovitch} and choose an at most countable disjoint subfamily~$\{I_i\}_i$ from our family~$\{I_t\}_{t\in (0,1]}$. Let those intervals~$I_i$ that correspond to the Lebesgue points be called~$g_i$, and those that correspond to non-Lebesgue points be called~$b_i$. Then,
\eq{
\sum_i|b_i| \leq \delta,\qquad \forall i\quad \av{\psi}{g_i} \in L.
}

It remains to apply a reasoning similar to the one presented in the proof of Theorem~\ref{Th1}. We express
\eq{
x=\av{\psi}{[0,1]}=\int\limits_{\cup_ib_i}\psi + \sum\limits_{i}|g_i|\sum\limits_{k}\frac{|g_k|}{\sum_{i}|g_i|}\av{\psi}{g_k}.
}
The first summand is small by the absolute continuity of~$\psi$, the second summand (disregarding the coefficient~$\sum_{i}|g_i|$) belongs to~$L$, and the coefficient~$\sum_{i}|g_i|$ is close to~$1$. This provides a contradiction to the fact that~$x\notin L$.
\end{proof}

\section{Proof of Theorem~\ref{Th2}}\label{S4}
\subsection{Proof of the splitting lemma}\label{s41}
We provide the proof of Lemma~\ref{EvenNewer5}. Without loss of generality (see Remark~\ref{InvarianceRemark}), assume~$I = [0,1]$. Consider two cases.

\paragraph{Case A: $x \notin \conv \OmOne$.} This case is similar to Lemma~\ref{Lemma5}. Consider the splitting~$I_+(t) = [0,t]$ and~$I_-(t) = [t,1]$ generated by~$t\in (0,1)$ and write~$x_{\pm} = x_{I_{\pm}}$ for brevity (recall~$x_J$ denotes~$\av{\psi}{J}$). Note that~$x_{\pm}$ are continuous functions of~$t$. We start with~$t = 1/2$. If~$[x_+,x_-] \cap \interior \conv \OmOne = \varnothing$ for~$t=1/2$, then we are done. In the other case, one of the segments~$[x_+,x)$ and~$(x,x_-]$ intersects~$\interior \conv \OmOne$. By convexity, \emph{at most} one of these segments can intersect~$\interior \conv \OmOne$. Without loss of generality, let~$[x_+,x)\cap \interior \conv \OmOne \ne \varnothing$.  Now, consider arbitrary~$t\in (0,1)$. Then,~$x_+(t) \to x$ as~$t \to 1$ and, therefore,~$[x_+(t),x]$ does not intersect~$\interior \conv \OmOne$ for~$t$ sufficiently close to~$1$. Let~$t_0$ be the smallest~$t > 1/2$ such that~$[x_+(t),x] \cap \interior \conv \OmOne = \varnothing$. In such a case,~$[x_+(t_0),x] \cap \conv\OmOne \ne \varnothing$, and, again by convexity,~$[x,x_-(t_0)]\cap \conv\OmOne = \varnothing$. We may set~$I_+ := I_+(t_0)$ and~$I_-:= I_-(t_0)$.

\paragraph{Case B: $x\in \conv \OmOne$.} In such a case,~$x\in L_j$ for some~$j$, by~\eqref{BoundaryStructure} and since~$x\notin \interior \conv \OmOne$. By Lemma~\ref{UsefulLemma}, there exist~$a$ and~$b$ such that~$0 < a < 1/2 < b < 1$ and
\eq{
A = \av{\psi}{[0,a]} \notin \omega_j,\qquad B =\av{\psi}{[b,1]} \notin \omega_j.
}
In what follows, we say that a point lies above (or below)~$L_j$ if it lies in the same (opposite) closed half-space as~$\OmOne$ with respect to~$L_j$; if we mean open half-space, then we say `strictly above' or `strictly below'. We skip the index~$j$ in our notation till the end of the proof for brevity. We wish to find~$t\in (0,1)$ such that~$\av{\psi}{[0,t]}\in L$; this will provide the desired splitting. Consider two subcases.

\paragraph{Subcase $1$: either~$A$ or~$B$ lie above~$L$.} Without loss of generality, assume~$A$ lies strictly above~$L$ (the case where~$A\in L$ is evident). Consider the curve~$\gamma \colon [a,1]\to \R^d$ given by the rule
\eq{
\gamma(t) = \av{\psi}{[0,t]},\qquad t\in [a,1].
}
The curve~$\gamma$ connects~$A$ to~$x$. Set
\eq{
t_0 = \inf \set{s \in [a,1]}{\gamma(s) \in L}.
}
It suffices to show~$t_0 < 1$. Assume the contrary, let~$t_0 = 1$. In such a case,~$\gamma(s)$ cannot lie below~$L$ for any~$s\in (a,1)$ since~$\gamma(a)$ lies strictly above. Then,~$\gamma(s)\in \omega$ for~$s$ sufficiently close to~$1$ by Remark~\ref{StructuralRemark}. Since~$A$ lies outside~$\omega$, there exists a point~$t$ such that~$\gamma(t) \in E \subset L$. This leads us to a contradiction and shows that~$t_0 \in [a,1)$. 

\paragraph{Subcase $2$: both~$A$ and~$B$ lie below~$L$.} In this case,~$\av{\psi}{[0,b]}$ is above~$L$, which yields there exists~$t_0\in [a,b]$ such that~$\av{\psi}{[0,t_0]}\in L$. We set~$I_+ = [0,t_0]$ and~$I_- = [t_0,1]$.\qed

\subsection{Transfinite Bellman induction}\label{s42}
Before passing to the proof of Theorem~\ref{Th2}, we introduce some notation. For any closed set~$S \subset [0,1]$ that contains~$0$ and~$1$, let~$\INT_S$ be the set of complementary intervals of~$S$.  If~$\psi$ is a summable function on~$[0,1]$ with values in~$\R^d$, we set
\eq{
\E{S}\psi(x) = \begin{cases}
\psi(x),\qquad x\in S;\\
\av{\psi}{I},\qquad x\in I\in \INT_S.
\end{cases}
}
In other words,~$\E{S} \psi$ is equal to the conditional expectation of~$\psi$ with respect to the~$\sigma$-field generated by the intervals in~$\INT_S$ and measurable subsets of~$S$. Note that if~$S_1 \subset S_2$ are closed sets, then
\eq{\label{Hereditary}
\E{S_1}\psi = \E{S_1}(\E{S_2}\psi).
}
\begin{Le}\label{ConvergenceLemma}
Let~$S_1 \subset S_2\subset S_3\ldots \subset S_n\subset \ldots$ be an increasing sequences of closed subsets of~$[0,1]$ \textup(as usual, we assume~$\{0,1\}\subset S_1$\textup). Let~$S= \cl(\cup_nS_n)$. Then\textup,~$\E{S_n}\psi(x)\to \E{S}\psi(x)$ for almost all~$x\in [0,1]$\textup, provided~$\psi$ is a summable function on~$[0,1]$.
\end{Le}
\begin{proof}
Let~$\Leb$ be the set of Lebesgue points of~$\E{S}\psi$. Consider several cases.

\paragraph{Case~$x\notin S$.} Then there exists~$I \in \INT_S$ such that~$x\in I$. There exist intervals~$I_n \in \INT_{S_n}$ that converge to~$I$. Then, by the continuity of the integral,
\eq{
\E{S_n}\psi(x) \Eeqref{Hereditary} \av{\psi}{I_n}\to \av{\psi}{I} = \E{S}\psi(x),\qquad n\to\infty.
}

\paragraph{Case~$x \in \cup_n S_n$.} In such a case,~$\E{S_n}\psi(x) = \E{S}\psi(x)$ for sufficiently large~$n$.

\paragraph{Case~$x\in S\setminus \cup_n S_n$.} First, we may disregard the points~$x\in [0,1]$ that are the endpoints of the intervals~$I\in \INT_S$ since there are at most countable number of them. Second, we may disregard the points~$x\notin \Leb$ since they form a set of measure zero. For the remaining points, there exists a sequence~$\{I_n\}_n$ such that for each~$n\in \mathbb{N}$ we have~$x\in I_n \in \INT_{S_n}$. Since~$S_{n} \subset S_{n+1}$, we also have~$I_{n+1}\subset I_n$. Therefore,~$\{I_n\}_n$ either shrinks to~$x$ or converges to an interval~$I \in \INT_S$. In the latter case,~$x$ is the endpoint of~$I$ since we have assumed~$x\in S$. In the former case, the desired convergence follows from~\eqref{Hereditary} and the fact~$x$ is the Lebesgue point of~$\E{S}\psi$:
\eq{
\E{S_n}\psi(x) \Eeqref{Hereditary}\E{S_n}(\E{S}\psi)(x) = \av{\E{S}\psi}{I_n} \to \E{S}\psi(x). 
} 
\end{proof}
Fix a function~$\psi$ as in Theorem~\ref{Th2}. Consider a set collection~$\F$:
\eq{
\F = \Set{S\subset [0,1]}{S \ \text{is closed},\ \{0,1\}\subset S,\ \forall I \in \INT_S \quad \av{\psi}{I}\notin \interior\conv\OmOne}.
} 
We will also use the quantity
\eq{
\ch S = \max\Set{|I|}{I\in \INT_S}.
}
\begin{Rem}\label{ClsdRm}
Let~$S_1,S_2,\ldots, S_n,\ldots$ be a set sequence as in Lemma~\textup{\ref{ConvergenceLemma},} let~$S = \cl (\cup_n S_n)$. If all~$S_n$ belong to~$\F$, then~$S\in \F$.  What is more\textup,~$\ch S = \lim_n \ch S_n$.
\end{Rem}
Before passing to the proof of Theorem~\ref{Th2}, we note that the functions~$\BG_{\OmConv,f}$ and~$\BG_{\OmNull \setminus \interior \conv \OmOne,f}$ coincide on their common domain (see Proposition~\ref{ContinuationProp} in Appendix). We denote the latter function by~$\BG$ for brevity.
\begin{proof}[Proof of Theorem~\textup{\ref{Th2}}.] We wish to apply Zorn's lemma and, thus, introduce a partial order on the set~$\F$. We will say that~$S_1\prec S_2$,~$S_1,S_2 \in \F$ if either~$S_1 = S_2$ or all the three conditions below hold:
\begin{enumerate}[1)]
\item $S_1 \subset S_2$;
\item $\ch S_2 < \ch S_1$;
\item $\av{\BG(\E{S_2}\psi)}{[0,1]} \leq \av{\BG(\E{S_1}\psi)}{[0,1]}$.
\end{enumerate}
To apply Zorn's lemma, we need to verify that every chain has an upper bound. Let~$\{\F_\alpha\}_{\alpha\in \mathfrak{A}}$ be a chain; it will be convenient to rename the indices: let~$\alpha := \ch S_\alpha$. Therefore,~$\mathfrak{A} \subset [0,1]$ and~$S_\alpha \prec S_{\beta}$, provided~$\alpha > \beta$.

Let~$S = \cl (\cup_\alpha S_{\alpha})$. If~$\inf \mathfrak{A}$ is attained at some~$\alpha\in \mathfrak{A}$, then there is nothing to prove. Since~$\ch S = \inf \mathfrak{A}$,~$\ch S < \ch S_{\alpha}$ for any~$\alpha \in \mathfrak{A}$. If we manage to prove that
\eq{\label{LimitingInequality1}
\av{\BG(\E{S}\psi)}{[0,1]} \leq \av{\BG(\E{S_\alpha}\psi)}{[0,1]} 
}
for any~$\alpha$, then~$S_{\alpha} \prec S$ for any~$\alpha$, and the assumptions of Zorn's lemma would be verified (note that~$S\in \F$ by Remark~\ref{ClsdRm}).   To show~\eqref{LimitingInequality1}, let~$\alpha_n$ be a decreasing downto~$\inf \mathfrak{A}$ sequence of indices~$\alpha$. It suffices to show
\eq{\label{MonotonicityLimit}
\av{\BG(\E{S}\psi)}{[0,1]} \leq \lim_{n}\av{\BG(\E{S_{\alpha_n}}\psi)}{[0,1]};
}
note that the limit on the right exists since this is the limit of a non-increasing sequence. The limit assertion~\eqref{MonotonicityLimit}, in its turn, follows from Lemma~\ref{ConvergenceLemma}, the lower semi-continuity of~$\BG$ (see Proposition~\ref{Continuity} in the Appendix), and Fatou's lemma:
\eq{
\av{\BG(\E{S}\psi)}{[0,1]} = \av{\BG(\lim_n\E{S_{\alpha_n}}\psi)}{[0,1]}\leq \av{\lim_n \BG(\E{S_{\alpha_n}}\psi)}{[0,1]} \leq  \lim_{n}\av{\BG(\E{S_{\alpha_n}}\psi)}{[0,1]}.
}
To apply Fatou's lemma, we need that~$\BG$ in non-negative. This follows from the fact~$f$ is non-negative and~\eqref{Third}, see Lemmas~$2.4$ and~$4.8$ in~\cite{SZ2022}.

By Zorn's lemma, there exists a maximal element~$S_{\max}$ in~$\F$ with respect to the chosen order. If~$S_{\max} \ne [0,1]$, the set~$\INT_{S_{\max}}$ is non-empty. Let~$I \in \INT_{S_{\max}}$ be an interval. If we apply Lemma~\ref{EvenNewer5} to it, we get a splitting~$I = I_+\cup I_-$ such that
\eq{
\BG\big(\av{\psi}{I}\big) \geq \frac{|I_{+}|}{|I|}\BG\big(\av{\psi}{I_+}\big) + \frac{|I_{-}|}{|I|}\BG\big(\av{\psi}{I_-}\big)
}
by local concavity of~$\BG$ on~$\OmNull \setminus \interior \conv\OmOne$. Thus, applying Lemma~\ref{EvenNewer5} to each of the largest intervals in~$\INT_{S_{\max}}$, we get a new element of~$\F$ that exceeds~$S_{\max}$, which contradicts its maximality. Thus,~$S_{\max} = [0,1]$. The definition of the order implies~\eqref{MainIneq}.
\end{proof}

\appendix
\section{Appendix}
\subsection{On Lebesgue points}\label{AS1}
We formulate a version of the Lebesgue differentiation theorem we need since it might seem unusual. Let~$\psi \colon \R \to \R^d$ be a locally summable function. We call a point~$t \in \R$ a Lebesgue point of~$\psi$, provided
\eq{
\lim\limits_{I \to t} \av{|\varphi - A|}{I} \to 0 
}
for some fixed number~$A = A(t)$. Here~$I$ is an arbitrary closed interval that contains~$t$. In particular,~$t$ might be the endpoint of~$I$. The convergence~$I \to t$ is the convergence of the endpoints of~$I$ to~$t$.
\begin{Th}[Lebesgue differentiation theorem]\label{OurLebesgue}
Almost all points are Lebesgue points and~$A(t) = \varphi(t)$ for almost all of them.
\end{Th} 
The details can be found in Section~$2.2$ of~\cite{Guzman1975}. In fact, Theorem~\ref{OurLebesgue} may be deduced in a standard way from the~$L_1 \to L_{1,\infty}$ boundedness of the corresponding maximal operator
\eq{
M \varphi(t) =\sup\limits_{I\colon t \in I}\av{|\varphi|}{I}.
}
This maximal operator admits a pointwise bound by the classical centered Hardy--Littlewood maximal operator, which justifies the required weak-type bound. 

We may set~$\psi(t) := A(t)$ and note that if~$t$ is a Lebesgue point of~$\psi$, then 
\eq{\label{A3}
\lim\limits_{I\to t}\av{\psi}{I} = \psi(t)
}
as a closed interval~$I$ containing~$x$ approaches~$x$. 

\subsection{Auxiliary geometric results}\label{AS2}
The proof of Lemma~\ref{ClosureConvexHull} will be based on a simple observation about convex bodies without rays on the boundary. This condition has already appeared in the related context in~\cite{AS2023}.
\begin{Le}\label{CompactCap}
Let $C\subset\mathbb{R}^d$ be a closed convex set and let $L$ be a hyperplane. If $C\cap L$ is non-empty and does not contain rays\textup, then $C\cap (L+\overline{B_r(0)})$ is a compact set for any $r\geqslant 0$.
\end{Le}
\begin{proof}
Since $C\cap L$ is not empty, there exists a point $x\in C\cap L$. The set~$C\cap(L+\overline{B_r(0)})$ is closed, so we only need to justify its boundedness. Assume the contrary: let there exist a sequence $\{y_n\}_n$ such that~$y_n\in C\cap(L+\overline{B_r(0)})$ and $|y_n|\to+\infty$. Let $\{y_{n_k}\}_k$ be a subsequence such that there exists the limit $e=\lim_{k}\frac{y_{n_k}}{|y_{n_k}|}$. From convexity and closedness of $C$, it follows that 
\eq{
\set{x+\alpha e}{\alpha\in[0,\,+\infty)}\subset C\cap L,
}
which contradicts our assumptions. Thus, $C\cap(L+\overline{B_r(0)})$ is a compact set.
\end{proof}
\begin{proof}[Proof of Lemma~\textup{\ref{ClosureConvexHull}}]
Assume the contrary: there exists a point~$x\in\partial\conv F\setminus\conv F$. Let~$L$ be a supporting hyperplane to~$\cl\conv F$ that contains~$x$. By Lemma~\ref{CompactCap},~$(\cl\conv F)\cap(L+\overline{B_r(0)})$ is compact. Therefore,~$F\cap(L+\overline{B_r(0)})$ is also a compact set. What is more, the sets
\eq{
H_r=\conv\Big(F\cap(L+\overline{B_r(0)})\Big)
}
are compact as well. On the other hand,~$H_r\subseteq\conv F$, which implies~$x \notin H_r$.
\begin{figure}[h]
\center{\includegraphics[scale=0.45]{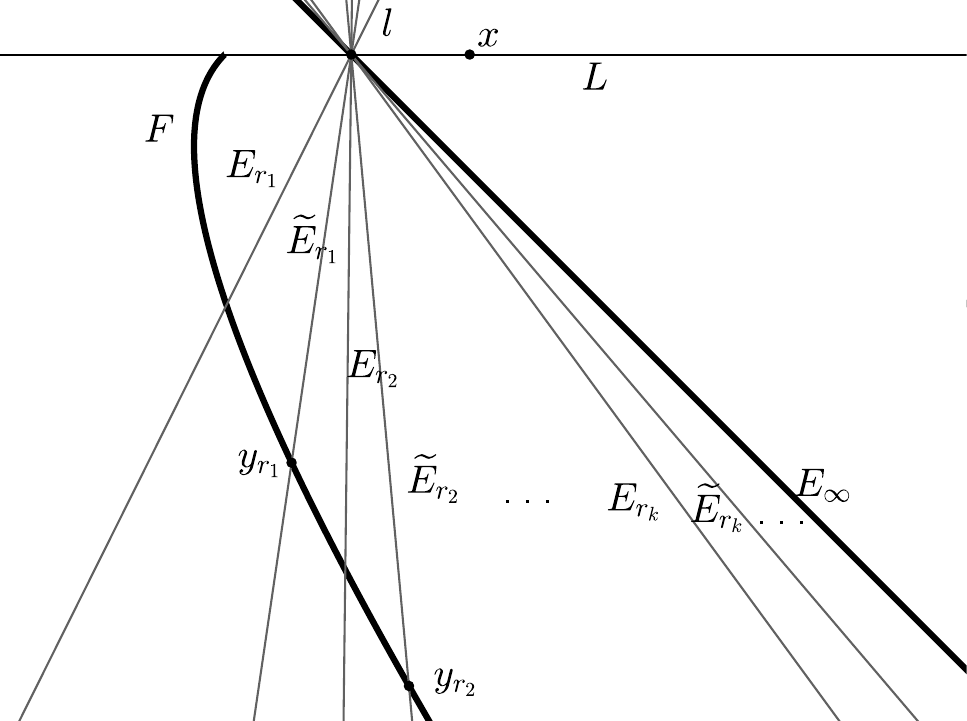}}
\caption{Illustration to the proof of Lemma~\ref{ClosureConvexHull}.}
\end{figure}

The set~$H_0$ is non-empty since~$F\cap L\ne\varnothing$ (if~$F\cap L =\varnothing$, then~$F\cap (L + \overline{B_r(0)}) = 0$ for some~$r > 0$, which contradicts~$x\in \overline{\conv F}$). Let~$\ell$ be a hyperplane of~$L$ that strictly separates~$x$ from~$H_0$. Let~$E_r$ be a hyperplane (of~$\R^d$) that contains~$\ell$ and strictly separates~$x$ from~$H_r$. Such a hyperplane exists since~$x\notin H_r$ and~$H_r$ is compact. Since~$x\in\cl\conv F$, the hyperplane~$E_r$ does not separate~$x$ from~$F$. Consequently, there exists a point~$y_r\in F$ such that~$x$ and~$y_r$ lie in the same half-space~$E_r^+$ with respect to~$E_r$. Since~$y_r\in F\setminus H_r$,~$|y_r|\to \infty$ as~$r\to \infty$.
Let~$\{r_k\}_k$ be a sequence such that $r_k\to+\infty$ and there exists the limit $e=\lim_{k}y_{r_k}/|y_{r_k}|$. Then,
\eq{
\set{x+\alpha e}{\alpha\in[0,\,+\infty)}\subset \cl\conv F.
}
It suffices to check that 
\eq{\label{A5}
\set{x+\alpha e}{\alpha\in[0,\,+\infty)}\subset L
}
to come to contradiction and finish the proof. 

Assume~\eqref{A5} fails. Then, the hyperplane
\eq{
E_{\infty} = \set{\alpha e}{\alpha\in\R}+\ell
} 
is not equal to~$L$ and~$x\notin E_{\infty}$. The hyperplane $L$ splits $\mathbb{R}^d$ into two half-spaces $L^+$ and $L^-$. Let $L^-$ be the half-space that contains~$F$. Then,~$H_r\subset E_r^-\cap L^-$. Let the hyperplane~$\widetilde{E}_{r_k}$ be the affine hull of~$\ell$ and~$y_{r_k}$ and let~$\widetilde{E}_{r_k}^+$ be the half-space that contains~$x$. Then $H_r\subset E_{r_k}^-\cap L^-\subset \widetilde{E}_{r_k}^-\cap L^-$. However,~$E_{\infty}$ is the limit of $\widetilde{E}_{r_k}$. Thus, we have  that $F\subset E_{\infty}^-$ and $x\in E_{\infty}^+$. Therefore,~$E_{\infty}$ strictly separates~$x$ from~$F$. This contradicts the fact that~$x\in\cl\conv F$.
\end{proof}

\begin{St}\label{Continuity}
Assume~$\OmOne$ is a convex set with non-empty interior. Assume~$f\colon \partial \OmNull \to \R$ is continuous. If~$\BG_{\OmNull \setminus \interior\OmOne}$ is finite\textup, then it is lower semi-continuous.
\end{St}
We will not provide a formal proof of Proposition~\ref{Continuity}. Note that a locally concave function is continuous at inner points of its domain. What is more, if~$f$ is continuous,~$\BG_{\OmNull\setminus \interior\OmOne}$ is also continuous at the points of~$\partial \OmNull$ by Proposition~$2.7$ in~\cite{SZ2016}. Thus, it remains to prove the lower semi-continuity at the points of~$\partial \OmOne$. In the case where~$\OmOne$ is strictly convex, the said continuity was proved in~\cite{SZ2016} (Proposition~$2.8$ of that paper). In fact, the lower and upper semi-continuities were proved separately, and the proof of lower semi-continuity did not use the strict convexity assumption, only the convexity of~$\OmOne$ was needed. Thus, the proof of Proposition~\ref{Continuity} is contained in~\cite{SZ2016}.

The proof of Proposition~$2.8$ in~\cite{SZ2016} does not use the minimality assumption: any finite locally concave function is continuous at the points of~$\partial\OmOne$, provided~$\partial\OmOne$ is strictly convex. This assertion fails if~$\OmOne$ is convex, but not necessarily strictly convex. We do not know whether~$\BG_{\OmNull \setminus \interior \OmOne}$ is continuous in this case.

\begin{St}\label{ContinuationProp}
Assume the set~$\OmOne$ is closed\textup, convex\textup, has non-empty interior\textup, and its boundary does not contain rays. Then\textup,
\eq{\label{FunctionAndItsExtension}
\BG_{\OmNull\setminus \OmOne}(x) = \BG_{\OmNull\setminus \interior \OmOne}(x)
}
for any~$x\in \OmNull\setminus \OmOne$.
\end{St}

\begin{figure}[h!]
\centerline{\includegraphics[height=8cm]{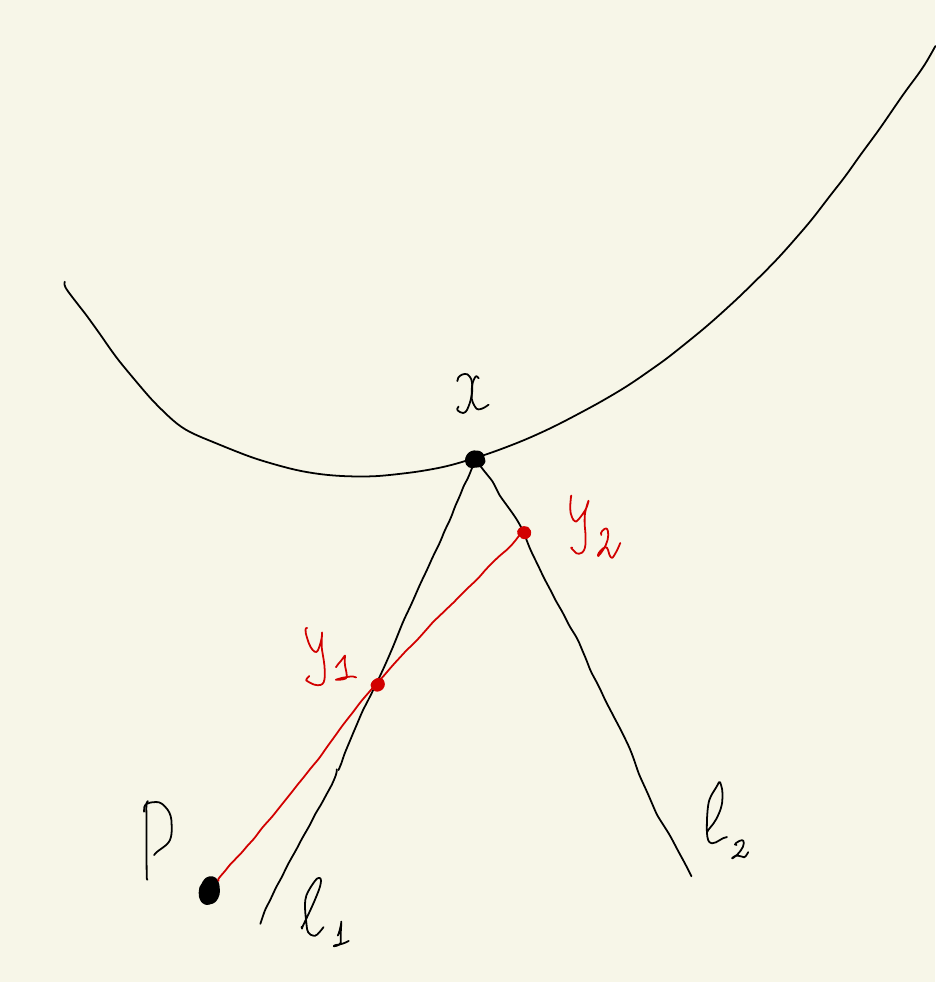}}
\caption{Illustration to the proof of Proposition~\ref{ContinuationProp}}
\label{Fig_TrS}
\end{figure}

\begin{proof} The proof is a minor modification of the proof of Proposition~$5.4$ in~\cite{SZ2022}. Let~$x$ be a point on the boundary of~$\OmOne$. Recall that a segment~$\ell \subset \Omega$ with the endpoint~$x$ is called transversal provided the continuation of~$\ell$ beyond~$x$ lies inside~$\OmOne$.  Let~$G$ be a finite locally concave function on~$\OmNull \setminus \OmOne$, let~$\ell_1$ and~$\ell_2$ be two transversal segments  with the endpoint $x$. We {\bf claim} that
\eq{\label{DefinitionOfExtension}
\lim\limits_{\genfrac{}{}{0pt}{-2}{y\to x}{y\in \ell_1}} G(y) = \lim\limits_{\genfrac{}{}{0pt}{-2}{y\to x}{y\in \ell_2}} G(y).
}
Note that both limits exist as limits of concave functions of single variables. Another issue is whether the limits are infinite; we will resolve it shortly.

To prove the claim, it suffices to justify the inequality
\eq{\label{LimitingInequality}
\lim\limits_{\genfrac{}{}{0pt}{-2}{y\to x}{y\in \ell_1}} G(y) \geq \lim\limits_{\genfrac{}{}{0pt}{-2}{y\to x}{y\in \ell_2}} G(y).
}
Consider the two-dimensional affine plane spanned by~$\ell_1$ and~$\ell_2$ and choose a point~$P$ as it is shown on Fig.~\ref{Fig_TrS}. Then,
\eq{
G(y_1) \geq \alpha G(y_2) + (1-\alpha)G(P),
}
for any~$y_1$ and~$y_2$ as on Fig.~\ref{Fig_TrS}. The coefficient~$\alpha$ tends to~$1$ as~$y_1 \to x$ along~$\ell_1$, which justifies~\eqref{LimitingInequality} and finishes the proof of the {\bf claim}.

Now we note that if the boundary of~$\OmOne$ does not contain rays, then for any~$x\in \partial \OmOne$ there exist~$y,z \in \interior\OmNull \setminus\OmOne$ such that~$x\in [y,z]$ and~$[y,z]\cap \interior \OmOne = \varnothing$. Let us show that the limit~\eqref{DefinitionOfExtension} is finite. Assume the transversal segment~$\ell_1$ is sufficiently short and consider the segments~$\ell_y$ and~$\ell_z$ that are its images under translations that map~$x$ to~$y$ and~$z$, correspondingly. Then, for any~$x_1 \in \ell_1$, we have
\eq{
G(x_1) \geq \frac{|y-x|}{|y-z|} G(z_1) + \frac{|z-x|}{|y-z|}G(y_1),
}
for some points~$y_1 \in \ell_y$ and~$z_1\in \ell_z$. This proves
\eq{\label{eq55}
\lim\limits_{\genfrac{}{}{0pt}{-2}{x_1\to x}{x_1\in \ell_1}} G(x_1) \geq \frac{|y-x|}{|y-z|} G(z) + \frac{|z-x|}{|y-z|}G(y).
}
In particular, the limit on the left is finite.

Now we are ready to prove~\eqref{FunctionAndItsExtension}. Note that
\eq{
\BG_{\OmNull\setminus \OmOne}(x) \geq \BG_{\OmNull\setminus \interior \OmOne}(x),\quad x\in \OmNull\setminus \OmOne.
}
To prove the reverse inequality, we extend the function~$\BG_{\OmNull\setminus  \OmOne}$ to the set~$\partial \OmOne$ using formula~\eqref{DefinitionOfExtension} (we plug this function for~$G$). It suffices to show the local concavity of the extension, which we denote by~$H$. We need to justify the inequality
\eq{
H(x) \geq \frac{|y-x|}{|y-z|} H(z) + \frac{|z-x|}{|y-z|}H(y),\qquad x\in [y,z]\subset \Om.
}
The inequality for arbitrary~$x,y,z$ follows from the cases where either~$[y,z]\subset \Om\setminus \partial \OmOne$ or~$x\in \partial\OmOne$, since concavity is a local notion. In the first case, the inequality is true since~$H|_{\OmNull\setminus \partial\OmOne}$ equals~$\BG_{\OmNull\setminus \OmOne}$. In the second case, the desired inequality is similar to~\eqref{eq55}. Thus,~$H$ is a locally concave function on the entire domain~$\OmNull\setminus \interior \OmOne$,~$H \geq \BG_{\OmNull\setminus \interior \OmOne}$, and~\eqref{FunctionAndItsExtension} holds true.
\end{proof}

\bibliography{/Users/mac/Documents/Bib/Mybib}{}
\bibliographystyle{plain}

Egor Dobronravov,

St. Petersburg State University,

yegordobronravov@mail.ru

\bigskip

Dmitriy Stolyarov,

St. Petersburg State University,

d.m.stolyarov@spbu.ru

\bigskip

Pavel Zatitskii,

University of Cincinnati,

zatitspl@ucmail.uc.edu

\end{document}